\documentclass{amsart}
\usepackage{bbding}
\usepackage{amssymb}
\usepackage{amsmath}

\newtheorem{theorem}{Theorem}[section]

\newtheorem{proposition}[theorem]{Proposition}
\newtheorem{corollary}[theorem]{Corollary}
\theoremstyle{definition}
\newtheorem{definition}[theorem]{Definition}

\newtheorem{remark}[theorem]{Remark}
\numberwithin{equation}{section}

\newcommand{\blankbox}[2]

\begin{document}
\title{On Schr\"{o}dinger-Virasoro type Lie conformal algebras}

\author{Yanyong Hong}
\address{College of Science, Zhejiang Agriculture and Forestry University,
Hangzhou, 311300, P.R.China}
\email{hongyanyong2008@yahoo.com}


\begin{abstract}
In this paper, two new classes of Schr\"{o}dinger-Virasoro type Lie conformal algebras $TSV(a,b)$ and $TSV(c)$ which are non-simple are introduced for some $a$, $b$, $c\in \mathbb{C}$. Moreover, central extensions, conformal derivations and  free conformal modules of rank 1 of $TSV(a,b)$ and $TSV(c)$ are determined.

\end{abstract}

\maketitle

\section{Introduction}
Lie conformal algebra introduced by V.Kac in \cite{K1}, \cite{K2} is an
axiomatic description of the operator product expansion (or rather
its Fourier transform) of chiral fields in conformal field theory.
It plays important roles in quantum field theory, vertex algebras and infinite-dimensional Lie algebras satisfying the locality property in \cite{K}. Moreover, Lie conformal algebras have close connections to Hamiltonian formalism in the theory of nonlinear evolution equations (see the book \cite{Do} and references therein, and also \cite{BDK, GD, Z, X4} and many other papers).

A Lie conformal algebra is said to be finite if it is finitely generated as a $\mathbb{C}[\partial]$-module. There are two important examples of finite Lie conformal algebras. One example is the Virasoro Lie conformal algebra $\text{Vir}$. It is defined by
$$\text{Vir}=\mathbb{C}[\partial]L, ~~[L_\lambda L]=(\partial+2\lambda)L.$$
It should be pointed out that a torsion-free Lie conformal algebra of rank 1 is either
trivial or isomorphic to $\text{Vir}$ (see \cite{DK1}).
The other example is the current Lie conformal algebra associated with a finite dimensional Lie algebra. Let $\mathfrak{g}$ be a finite dimensional Lie algebra. The current Lie conformal
algebra associated to $\mathfrak{g}$ is defined by:
$$\text{Cur} \mathfrak{g}=\mathbb{C}[\partial]\otimes \mathfrak{g}, ~~[a_\lambda b]=[a,b],
~~a,b\in \mathfrak{g}.$$
It is shown in \cite{DK1} that $\text{Vir}$ and all current Lie conformal algebras $\text{Cur} \mathfrak{g}$ where $\mathfrak{g}$ is a finite dimensional simple Lie algebra exhausts all finite simple Lie conformal algebras. Moreover, all irreducible finite conformal modules of finite simple Lie conformal algebras are determined in \cite{CK1}
and cohomology groups of finite simple Lie conformal algebras with some conformal modules are characterized in \cite{BKV}.

As for finite non-simple Lie conformal algebras, there are few examples and related results. Only recently, Su and Yuan in \cite{SY} investigate two new non-simple Lie conformal algebras which are obtained from Schr\"{o}dinger-Virasoro Lie algebra and the extended Schr\"{o}dinger-Virasoro Lie algebra. Similarly,  a Schr\"{o}dinger-Virasoro type Lie conformal algebra obtained from a twisted case of the deformative Schr\"{o}dinger-Virasoro Lie algebra are studied in \cite{WXX}. In fact, their methods are the same, i.e. to construct Lie conformal algebras from some known formal distribution Lie algebras (see \cite{K1}). In addition, a classification of the torsion-free Lie conformal algebras of rank 2 is presented in \cite{HL}. In this paper, we plan to present a generalization of Schr\"{o}dinger-Virasoro Lie conformal algebra in \cite{SY} and the Schr\"{o}dinger-Virasoro type Lie conformal algebra introduced in \cite{WXX} from the point of view of Lie conformal algebra. Through this generalization, we obtain two new classes of non-simple Lie conformal algebras of rank 3, i.e. $TSV(a,b)$ and $TSV(c)$.  Moreover, central extensions, conformal derivations and  free conformal modules of rank 1 of $TSV(a,b)$ and $TSV(c)$ are determined. These results will enrich the theory of finite non-simple Lie conformal algebra.

This paper is organized as follows. In Section 2, we introduce some definitions, notations and some results about Lie conformal algebras.
In Section 3, two new classes of Schr\"{o}dinger-Virasoro type Lie conformal algebras $TSV(a,b)$ and $TSV(c)$ are presented. In Section 4,
central extensions of $TSV(a,b)$ and $TSV(c)$ by a one-dimensional center are determined. In Section 5, we characterize conformal derivations of
$TSV(a,b)$ and $TSV(c)$. Moreover, free conformal modules of rank one of $TSV(a,b)$ and $TSV(c)$ are investigated.

Throughout this paper, denote by $\mathbb{C}$ the field of complex
numbers; $\mathbf{N}$ is the set of natural numbers, i.e.,
$\mathbf{N}=\{0, 1, 2,\cdots\}$; $\mathbb{Z}$ is the set of integer
numbers.

\section{Preliminaries}
In this section, we will recall some definitions, notations and some results about Lie conformal algebras. These facts can be referred to \cite{K1}.
\begin{definition}
A \emph{Lie conformal algebra} $R$ is a $\mathbb{C}[\partial]$-module with a $\lambda$-bracket $[\cdot_\lambda \cdot]$ which defines a $\mathbb{C}$-bilinear
map from $R\otimes R\rightarrow R[\lambda]$, where $R[\lambda]= R\otimes\mathbb{C}[\lambda]$ is the space of polynomials of $\lambda$ with coefficients
in $R$, satisfying
\begin{eqnarray*}
&&[\partial a_\lambda b]=-\lambda [a_\lambda b],~~~[a_\lambda \partial b]=(\lambda+\partial)[a_\lambda b], ~~\text{(conformal sesquilinearity)}\\
&&[a_\lambda b]=-[b_{-\lambda-\partial}a],~~~~\text{(skew-symmetry)}\\
&&[a_\lambda[b_\mu c]]=[[a_\lambda b]_{\lambda+\mu} c]+[b_\mu[a_\lambda c]],~~~~~~\text{(Jacobi identity)}
\end{eqnarray*}
for $a$, $b$, $c\in R$.
\end{definition}

A  Lie conformal algebra
is called \emph{finite} if it is finitely generated as a
$\mathbb{C}[\partial]$-module. The \emph{rank} of a Lie conformal
algebra $R$ is its rank as a $\mathbb{C}[\partial]$-module. We say a Lie
conformal algebra $R$ is \emph{torsion-free} just means that $R$ as
a $\mathbb{C}[\partial]$-module is torsion-free. Since
$\mathbb{C}[\partial]$ is a principle ideal domain, a finite
torsion-free Lie conformal algebra $R$ is free as a
$\mathbb{C}[\partial]$-module.

In addition, given a Lie conformal algebra $R$, there is an important infinite-dimensional Lie algebra associated with it.
Set $[a_\lambda b]=\sum_{n\in \mathbf{N}}\frac{\lambda^n}{n!}a_{(n)}b$. Let Coeff$(R)$ be the quotient
of the vector space with basis $a_n$ $(a\in R, n\in\mathbb{Z})$ by
the subspace spanned over $\mathbb{C}$ by
elements:
$$(\alpha a)_n-\alpha a_n,~~(a+b)_n-a_n-b_n,~~(\partial
a)_n+na_{n-1},~~~\text{where}~~a,~~b\in R,~~\alpha\in \mathbb{C},~~n\in
\mathbb{Z}.$$ The operation on Coeff$(R)$ is defined as follows:
\begin{equation}\label{106}
[a_m, b_n]=\sum_{j\in \mathbf{N}}\left(\begin{array}{ccc}
m\\j\end{array}\right)(a_{(j)}b)_{m+n-j}.\end{equation} Then,
Coeff$(R)$ is a Lie algebra and it is called the\emph{ coefficient algebra} of $R$ (see \cite{K1}).

For example, $\text{Coeff}(\text{Vir})$ is isomorphic to Witt algebra and $\text{Coeff}(\text{Cur}\mathfrak{g})$ is just the loop algebra associated with $\mathfrak{g}$.

In fact, $\text{Vir}$ and $\text{Cur}\mathfrak{g}$ belong to a class of special Lie conformal algebras named quadratic Lie conformal algebras (see \cite{X1}). Next, we introduce the definition of quadratic Lie conformal algebra.

\begin{definition}
$R$ is a \emph{quadratic Lie conformal algebra}, if $R=\mathbb{C}[\partial]V$ is a Lie conformal algebra as a free
$\mathbb{C}[\partial]$-module and the $\lambda$-bracket is of the following form:
$$[a_{\lambda} b]=\partial u+\lambda v+ w,$$
where $a$, $b$, $u$, $v$, $w\in V$.
\end{definition}

For giving an equivalent characterization of quadratic Lie conformal algebra, we present the definitions of Novikov algebra and
Gel'fand-Dorfman bialgebra.
\begin{definition}
A \emph{Novikov algebra} $A$ is a vector space over $\mathbb{C}$ with a bilinear product $\circ: A\times A\rightarrow A$ satisfying (for any $a$, $b$, $c\in A$):
\begin{eqnarray}
&&(a\circ b)\circ c-a \circ (b\circ c)=(b\circ a)\circ c-b \circ (a\circ c),\\
&&(a\circ b)\circ c=(a\circ c)\circ b.
\end{eqnarray}
\begin{remark}
Novikov algebra was essentially stated in \cite{GD} that it corresponds to a
certain Hamiltonian operator. Moreover, it also appeared
in \cite{BN} from the point of view of Poisson structures of
hydrodynamic type. The name ``Novikov algebra" was given by Osborn
in \cite{Os}.
\end{remark}
\end{definition}

\begin{definition}(see \cite{GD} or \cite{X1})
A \emph{Gel'fand-Dorfman bialgebra} $A$ is a vector space over $\mathbb{C}$ with two algebraic operations $[\cdot,\cdot]$ and $\circ$ such that $(A,[\cdot,\cdot])$ forms a Lie algebra, $(A,\circ)$ forms a Novikov algebra and the following compatibility condition holds:
\begin{eqnarray}
[a\circ b, c]-[a\circ c, b]+[a,b]\circ c-[a,c]\circ b-a\circ [b,c]=0,
\end{eqnarray}
for $a$, $b$, and $c\in A$. We usually denote it by $(A,\circ,[\cdot,\cdot])$.
\end{definition}

An equivalent characterization of quadratic Lie conformal algebra is given as follows.
\begin{theorem}(see \cite{GD} or \cite{X1})
$R=\mathbb{C}[\partial]V$ is a quadratic Lie conformal algebra if and only if the $\lambda$-bracket of $R$ is given as
follows
$$[a_{\lambda} b]=\partial(b\circ a)+[b, a]+\lambda(b\ast a), \text{where $a\ast b=a\circ b+b\circ a$ for $a$, $b\in V$},$$
and $(V, \circ, [\cdot,\cdot])$ is a Gel'fand-Dorfman bialgebra. Therefore, $R$ is called the quadratic Lie conformal algebra corresponding to
the Gel'fand-Dorfman bialgebra $(V, \circ, [\cdot,\cdot])$.
\end{theorem}

Finally, we introduce the definition of conformal module.
\begin{definition}
A \emph{conformal module} $M$ over a Lie conformal algebra $R$ is a $\mathbb{C}[\partial]$-module endowed with the $\lambda$-action $a_\lambda v$ which defines
a $\mathbb{C}$-bilinear map $R\otimes M\rightarrow M[\lambda]$ such that
\begin{eqnarray*}
&&(\partial a)_\lambda v=-\lambda a_\lambda v,~~~a_\lambda (\partial v)=(\lambda+\partial)a_\lambda v,\\
&& a_\lambda(b_\mu v)-b_\mu(a_\lambda v)=[a_\lambda b]_{\lambda+\mu}v,
\end{eqnarray*}
for all $a$, $b\in R$, $v\in M$.
\end{definition}
If $M$ is finitely generated over $\mathbb{C}[\partial]$, then $M$ is called \emph{finite}.
The \emph{rank} of a Lie conformal
algebra $R$ is its rank as a $\mathbb{C}[\partial]$-module.

In the following, we will simply call a ``conformal module" a ``module". And, for any Lie conformal algebra $R$, ``$R$-module" just means
``conformal module of $R$".

For the Virasoro Lie conformal algebra $\text{Vir}$, it is known from \cite{CK1} that
\begin{proposition}\label{PP1}
All free non-trivial $\text{Vir}$-modules of rank one over $\mathbb{C}[\partial]$ are as follows ($\alpha$, $\beta\in \mathbb{C}$):
\begin{eqnarray*}
M_{\alpha,\beta}=\mathbb{C}[\partial]v,~~L_\lambda v=(\partial+\alpha\lambda+\beta)v.
\end{eqnarray*}
 This module is irreducible if and only if
$\alpha$ is non-zero, and all irreducible  $\text{Vir}$-modules are of this kind.
\end{proposition}

\section{Two classes of Schr\"{o}dinger-Virasoro type Lie conformal algebras}
In this section, we will introduce two new classes of Schr\"{o}dinger-Virasoro type Lie conformal algebras from the point view of Lie conformal algebra.

First, let us recall the Schr\"{o}dinger-Virasoro Lie conformal algebra (see \cite{SY}) and the Schr\"{o}dinger-Virasoro type Lie conformal algebra introduced in \cite{WXX}.

The Schr\"{o}dinger-Virasoro Lie conformal algebra is $SV=\mathbb{C}[\partial]L\oplus\mathbb{C}[\partial]M\oplus\mathbb{C}[\partial]Y$ with the following non-trivial $\lambda$-brackets:
\begin{eqnarray}
&&[L_\lambda L]=(\partial+2\lambda)L, ~~~[L_\lambda Y]=(\partial+\frac{3}{2}\lambda)Y,\\
&&[Y_\lambda Y]=(\partial+2\lambda)M,~~~[L_\lambda M]=(\partial+\lambda)M.
\end{eqnarray}
It is obtained from the Schr\"{o}dinger-Virasoro Lie algebra.

The Schr\"{o}dinger-Virasoro type Lie conformal algebra introduced in \cite{WXX} is $DSV=\mathbb{C}[\partial]L\oplus\mathbb{C}[\partial]M\oplus\mathbb{C}[\partial]Y$ with the following
non-trivial $\lambda$-brackets:
\begin{eqnarray}
&&[L_\lambda L]=(\partial+2\lambda)L, ~~~[L_\lambda Y]=\partial Y,\\
&&[Y_\lambda Y]=(\partial+2\lambda)M,~~~[L_\lambda M]=(\partial-2\lambda)M.
\end{eqnarray}

Obviously, $SV$ and $DSV$ are two non-simple torsion-free Lie conformal algebras of rank 3. Moreover, in both algebras, $\mathbb{C}[\partial]L$ is the Virasoro Lie conformal algebra, and $\mathbb{C}[\partial]M$ and $\mathbb{C}[\partial]Y$ are non-trivial modules of $\mathbb{C}[\partial]L$. Base on these facts, we begin to give a generalization of
$SV$ and $DSV$.

Set $TSV=\mathbb{C}[\partial]L\oplus\mathbb{C}[\partial]M\oplus\mathbb{C}[\partial]Y$. Let $\mathbb{C}[\partial]L$ be the Virasoro Lie conformal algebra. Assume that $\mathbb{C}[\partial]M$ and $\mathbb{C}[\partial]Y$ are non-trivial modules of $\mathbb{C}[\partial]L$. According to $SV$ and $DSV$, we also set $[Y_\lambda M]=[M_\lambda M]=0$. Moreover, with these assumptions and by
Proposition \ref{PP1}, we set the $\lambda$-brackets on $TSV$ as follows:
\begin{eqnarray}
\label{e31}&&[L_\lambda L]=(\partial+2\lambda)L,~~~~[L_\lambda Y]=(\partial+\alpha_1\lambda +\alpha_2)Y,\\
\label{e32}&&[L_\lambda M]=(\partial+\beta_1\lambda +\beta_2)M,~~~[Y_\lambda Y]=(\partial+\gamma_1\lambda +\gamma_2)M,\\
\label{e33}&&[Y_\lambda M]=[M_\lambda M]=0,
\end{eqnarray}
for some $\alpha_1$, $\alpha_2$, $\beta_1$, $\beta_2$, $\gamma_1$, $\gamma_2\in \mathbb{C}$.

\begin{theorem}
$TSV$ is a Lie conformal algebra with the $\lambda$-brackets given by
(\ref{e31})-(\ref{e33}) if and only if $\gamma_1=2$, $\gamma_2=0$, $\beta_2=2\alpha_2$ and $\beta_1=2(\alpha_1-1)$.
\end{theorem}
\begin{proof}
According to $[Y_\lambda Y]=-[Y_{-\lambda-\partial}Y]$, it is easy to obtain $\gamma_1=2$ and $\gamma_2=0$.
Then, by the assumption on $TSV$, we only need to check that whether the Jacobi identity holds when $a
=L$, $b=c=Y$.
Taking $a=L$, $b=Y$, $c=Y$ in the Jacobi identity and using (\ref{e31})-(\ref{e33}), by comparing the coefficients of $M$, we obtain
\begin{gather}
\label{e34}
(\partial+\lambda+2\mu)(\partial+\beta_1\lambda+\beta_2)=((\alpha_1-1)\lambda-\mu+\alpha_2)(\partial+2\lambda+2\mu)\\
+(\partial+\alpha_1\lambda+\mu+\alpha_2)(\partial+2\mu).
\nonumber
\end{gather}
By comparing the coefficients of $\partial^2$, $\lambda\partial$, $\mu\partial$, $\lambda^2$, $\mu^2$, $\lambda\mu$, $\lambda$, $\mu$,
$\partial$, $\lambda^0$, (\ref{e34}) holds if and only if $\beta_2=2\alpha_2$, $\beta_1=2(\alpha_1-1)$.
Thus, this theorem holds.
\end{proof}

Then, for the readers' convenience, we present the following definition in detail.
\begin{definition}
$TSV(a,b)=\mathbb{C}[\partial]L\oplus\mathbb{C}[\partial]M\oplus\mathbb{C}[\partial]Y$ is a Lie conformal algebra with the following $\lambda$-brackets:
\begin{eqnarray}
&&[L_\lambda L]=(\partial+2\lambda)L,~~~~[L_\lambda Y]=(\partial+a\lambda +b)Y,\\
&&[L_\lambda M]=(\partial+2(a-1)\lambda +2b)M,~~~[Y_\lambda Y]=(\partial+2\lambda)M,\\
&&[Y_\lambda M]=[M_\lambda M]=0,
\end{eqnarray}
for some $a$, $b\in \mathbb{C}$.
\end{definition}
\begin{remark}
When $a=\frac{3}{2}$, $b=0$, $TSV(\frac{3}{2},0)$ is just the Schr\"{o}dinger-Virasoro Lie conformal algebra studied in \cite{SY}; when $a=b=0$, $TSV(0,0)$ is just the Schr\"{o}dinger-Virasoro type Lie conformal algebra studied in \cite{WXX}.
\end{remark}
\begin{remark}\label{r1}
Obviously, $TSV(a,b)$ is a quadratic Lie conformal algebra corresponding to the 3-dimensional Gel'fand-Dorfman bialgebra
$V(a,b)=\mathbb{C}L\oplus\mathbb{C}Y\oplus\mathbb{C}M$ with the Novikov operation ``$\circ$" and the Lie bracket defined as follows:
\begin{eqnarray}
&&L\circ L=L,~~~L\circ Y=(a-1)Y,~~~Y\circ L=Y,~~~L\circ M=(2a-3)M,\\
&&M\circ L=M,~~Y\circ Y=M,~~Y\circ M=M\circ Y=M\circ M=0,\\
&&[L,L]=[Y,Y]=[Y,M]=[M,M]=0,~~~[L,Y]=-b Y,~~~[L,M]=-2bM.
\end{eqnarray}
\end{remark}

\begin{remark}\label{rr}
In fact, in a general case, we can assume that $[Y_\lambda Y]=P(\lambda,\partial)M$ in (\ref{e32}) for some non-zero element $P(\lambda,\partial)\in \mathbb{C}[\lambda,\partial]$. By the skew-symmetry, we obtain
\begin{eqnarray}
\label{e36}P(\lambda,\partial)=-P(-\lambda-\partial,\partial).
\end{eqnarray}
Let $Q(\lambda,\partial)=P(\lambda,-\lambda-\partial)$. By (\ref{e36}), we get $Q(\lambda,-\lambda-\partial)=-Q(-\lambda-\partial,\lambda)$. Then, letting $x=\lambda$, $y=-\lambda-\partial$,  $Q(x,y)=-Q(y,x)$. Therefore, $Q(x,y)=(x-y)S(x,y)$, where $S(x,y)$ is a symmetric polynomial.
Then, by the Jacobi identity holds when $a=L$, $b=Y$, $c=Y$, we obtain
\begin{gather}
\label{e37}
S(\mu,-\lambda-\mu-\partial)(\partial+\lambda+2\mu)(\partial+\beta_1\lambda+\beta_2)
=((\alpha_1-1)\lambda-\mu+\alpha_2)(\partial+2\lambda+2\mu)\cdot\\
S(\lambda+\mu,-\lambda-\mu-\partial)
+(\partial+\alpha_1\lambda+\mu+\alpha_2)(\partial+2\mu)S(\mu,-\mu-\partial).\nonumber
\end{gather}
However, it seems
impossible to give all solutions of the equation directly.
But, when the degree $m$ of $S(x,y)$ is fixed, by comparing the coefficients of
$\partial^i$ for $i\in [0,m+2]$, we can determine all solutions of (\ref{e37}). Of course, when $m$ is large, the
computation is complicated.

For example, when $S(x,y)$ is a symmetric ploynomial of degree one, we set $S(x,y)=x+y+b$ where $b\in \mathbb{C}$.
Then, taking it into (\ref{e37}), by some computations, we can immediately obtain
$\beta_1=0$, $\beta_2=2\alpha_2$, $\alpha_1=\frac{3}{2}$, $b=-2\alpha_2$. Therefore,
$P(\lambda,\partial)=(\partial+2\lambda)(-\partial-2\alpha_2)$.
\end{remark}

By the above discussion in Remark \ref{rr}, we obtain the following Lie conformal algebra.
\begin{definition}
$TSV(c)=\mathbb{C}[\partial]L\oplus\mathbb{C}[\partial]M\oplus\mathbb{C}[\partial]Y$ is a Lie conformal algebra with the following $\lambda$-brackets:
\begin{eqnarray}
&&[L_\lambda L]=(\partial+2\lambda)L,~~~~[L_\lambda Y]=(\partial+\frac{3}{2}\lambda +c)Y,\\
&&[L_\lambda M]=(\partial+2c)M,~~~[Y_\lambda Y]=(\partial+2\lambda)(-\partial-2c)M,\\
&&[Y_\lambda M]=[M_\lambda M]=0,
\end{eqnarray}
for some $c\in \mathbb{C}$.
\end{definition}

\begin{remark}
By the definition of coefficient algebra of a Lie conformal algebra,
$\text{Coeff}(TSV(a,b))=\oplus_{i\in \mathbb{Z}}\mathbb{C}L_i\oplus_{i\in \mathbb{Z}}\mathbb{C}Y_i\oplus_{i\in \mathbb{Z}}\mathbb{C}M_i$  is an infinite-dimensional Lie algebra with the following Lie brackets (the other Lie brackets vanishing):
\begin{eqnarray*}
&&[L_m,L_n]=(m-n)L_{m+n-1},~~~[L_m,Y_n]=(m(a-1)-n)Y_{m+n-1}+bY_{m+n},\\
&&[L_m,M_n]=(m(2a-3)-n)M_{m+n-1}+2bM_{m+n},~~~~[Y_m,Y_n]=(m-n)M_{m+n-1}.
\end{eqnarray*}

Similarly, $\text{Coeff}(TSV(c))=\oplus_{i\in \mathbb{Z}}\mathbb{C}L_i\oplus_{i\in \mathbb{Z}}\mathbb{C}Y_i\oplus_{i\in \mathbb{Z}}\mathbb{C}M_i$ is an infinite-dimensional Lie algebra with the following Lie brackets (the other Lie brackets vanishing):
\begin{eqnarray*}
&&[L_m,L_n]=(m-n)L_{m+n-1},~~~[L_m,Y_n]=(\frac{m}{2}-n)Y_{m+n-1}+cY_{m+n},\\
&&[L_m,M_n]=-(m+n)M_{m+n-1}+2cM_{m+n},~~~~\\
&&[Y_m,Y_n]=(m-n)(m+n-1)M_{m+n-2}+2c(n-m)M_{m+n-1}.
\end{eqnarray*}

\section{Central extensions of $TSV(a,b)$ and $TSV(c)$}
In this section, we will study central extensions of $TSV(a,b)$ and $TSV(c)$ by a one-dimensional center $\mathbb{C}\mathfrak{c}$.

An extension of a Lie conformal algebra $R$ by an abelian Lie conformal algebra $C$ is a short exact sequence of Lie conformal algebras
\begin{eqnarray*}
0\rightarrow C\rightarrow \widehat{R}\rightarrow R\rightarrow 0.
\end{eqnarray*}
$\widehat{R}$ is called an \emph{extension} of $R$ by $C$ in this case. This extension is called \emph{central} if $\partial C=0$ and
$[C_\lambda \widehat{R}]=0$.

In the following, we focus on the central extension $\widehat{R}$ of $R$ by a one-dimensional center $\mathbb{C}\mathfrak{c}$. This implies that $\widehat{R}=R\oplus \mathbb{C}\mathfrak{c}$, and
\begin{eqnarray*}
[a_\lambda b]_{\widehat{R}}=[a_\lambda b]_R+\alpha_\lambda(a,b)\mathfrak{c}, \text{for~~all~~$a$, $b\in R$,}
\end{eqnarray*}
where $\alpha_\lambda: R\times R\rightarrow \mathbb{C}[\lambda]$ is a $\mathbb{C}$-bilinear map. By the axioms of Lie conformal algebra, $\alpha_\lambda$ should satisfy the following properties (for all $a$, $b$, $c\in R$) :
\begin{eqnarray}
&&\label{ef3}\alpha_\lambda(\partial a,b)=-\lambda \alpha_\lambda(a,b)=-\alpha_\lambda(a,\partial b),\\
&&\label{ef4}\alpha_\lambda(a,b)=-\alpha_{-\lambda}(b,a),\\
&&\label{ec1}\alpha_\lambda(a,[b_\mu c])-\alpha_\mu(b,[a_\lambda c])=\alpha_{\lambda+\mu}([a_\lambda b],c).
\end{eqnarray}

Since $TSV(a,b)$ is a quadratic Lie conformal algebra, for studying its central extensions, we recall a proposition.
\begin{proposition}\label{rr1}(see \cite{H})
Let $\widehat{R}=R\oplus\mathbb{C}\mathfrak{c}$ be a central extension of quadratic Lie conformal algebra $R=\mathbb{C}[\partial]V$  corresponding to $(V,\circ, [\cdot,\cdot])$ by a one-dimensional center $\mathbb{C}\mathfrak{c}$.
Set the $\lambda$-bracket of $\widehat{R}$ by
\begin{eqnarray}
\widetilde{[a_\lambda b]}=\partial(b\circ a)+\lambda(a\ast b)+[b,a]+\alpha_\lambda(a,b)\mathfrak{c},
\end{eqnarray}
where $a$, $b\in V$ and $\alpha_\lambda(a,b)\in \mathbb{C}[\lambda]$. Assume that $\alpha_\lambda(a,b)=\sum_{i=0}^n\lambda^i\alpha_i(a,b)$ for any $a$, $b\in V$, where there exist some $a$, $b \in V$ such that
$\alpha_n(a,b)\neq 0$. Then, we obtain (for any $a$, $b$, $c\in V$)\\
(1) If $n>3$,  $\alpha_n(a\circ b, c)=0$ ;\\
(2) If $n\leq 3$,
\begin{eqnarray}
\label{eqqq1}&&\alpha_i(a,b)=(-1)^{i+1}\alpha_i(b,a), \text{ for any $i\in\{0,1,2,3\}$}, \\
\label{eqqe1}&&\alpha_3(a,c\circ b)=\alpha_3(a\circ b,c)=\alpha_3(b\circ a, c),\\
\label{eqq2}&&\alpha_2(a,c\circ b)+\alpha_3(a,[c,b])=\alpha_2(a\circ b,c)+\alpha_3([b,a],c),\\
\label{eqq3}&&\alpha_2(a,b\ast c)+\alpha_2(b\circ a,c)=2\alpha_2(a\circ b,c)+3\alpha_3([b,a],c),\\
\label{eqq4}&&\alpha_1(a,c\circ b)+\alpha_2(a,[c,b])=\alpha_1(a\circ b,c)+\alpha_2([b,a],c),\\
\label{eqq5}&&\alpha_1(a,b\ast c)-\alpha_1(b,a\ast c)=-\alpha_1(b\circ a,c)+\alpha_1(a\circ b,c)+2\alpha_2([b,a],c),\\
\label{eqq6}&&\alpha_0(a,c\circ b)+\alpha_1(a,[c,b])-\alpha_0(b, a\ast c)=\alpha_0(a\circ b,c)+\alpha_1([b,a],c),\\
\label{eqq7}&&\alpha_0(a,[c,b])-\alpha_0(b,[c,a])=\alpha_0([b,a],c).
\end{eqnarray}
\end{proposition}

Then, central extensions of $TSV(a,b)$ by a one-dimensional center $\mathbb{C}\mathfrak{c}$ are characterized as follows.
\begin{theorem}\label{th2}
(1) If $a\neq 1$ and $a\neq 2$ or $b\neq 0$, any central extension of $TSV(a,b)$ by a one-dimensional center $\mathbb{C}\mathfrak{c}$ is of the following form: $\widetilde{TSV(a,b)}=TSV(a,b)\oplus \mathbb{C}\mathfrak{c}$ and the non-trivial $\lambda$-brackets are given as follows:
\begin{eqnarray}
&&[L_\lambda L]=(\partial+2\lambda)L+(A\lambda+B\lambda^3)\mathfrak{c},\\
&&[L_\lambda Y]=(\partial+a\lambda+b)Y+(C+D\lambda)\mathfrak{c},\\
&&[Y_\lambda Y]=(\partial+2\lambda)M+E\lambda \mathfrak{c},\\
&&[L_\lambda M]=(\partial+2(a-1)\lambda+2b)M+(bE+(a-1)E\lambda)\mathfrak{c},
\end{eqnarray}
where $A$, $B$, $C$, $D$, $E\in \mathbb{C}$ and $aC=bD$.\\

(2) If $a=2$ and $b=0$, any central extension of $TSV(2,0)$ by a one-dimensional center $\mathbb{C}\mathfrak{c}$ is of the following form: $\widetilde{TSV(2,0)}=TSV(2,0)\oplus \mathbb{C}\mathfrak{c}$ and the non-trivial $\lambda$-brackets are given as follows :
\begin{eqnarray}
&&[L_\lambda L]=(\partial+2\lambda)L+(A\lambda+B\lambda^3)\mathfrak{c},\\
&&[L_\lambda Y]=(\partial+2\lambda)Y+(C\lambda+D\lambda^3)\mathfrak{c},\\
&&[Y_\lambda Y]=(\partial+2\lambda)M+(E\lambda+F\lambda^3) \mathfrak{c},\\
&&[L_\lambda M]=(\partial+2\lambda)M+(E\lambda+F\lambda^3)\mathfrak{c},
\end{eqnarray}
where $A$, $B$, $C$, $D$, $E$, $F\in \mathbb{C}$.\\

(3) If $a=1$ and $b=0$, any central extension of $TSV(1,0)$ by a one-dimensional center $\mathbb{C}\mathfrak{c}$ is of the following form: $\widetilde{TSV(1,0)}=TSV(1,0)\oplus \mathbb{C}\mathfrak{c}$ and the non-trivial $\lambda$-brackets are given as follows:
\begin{eqnarray}
&&[L_\lambda L]=(\partial+2\lambda)L+(A\lambda+B\lambda^3)\mathfrak{c},\\
&&[L_\lambda Y]=(\partial+\lambda)Y+(C\lambda+D\lambda^2)\mathfrak{c},~~[L_\lambda M]=\partial M,\\
&&[Y_\lambda Y]=(\partial+2\lambda)M+E\lambda \mathfrak{c},~~[Y_\lambda M]=F\mathfrak{c}.
\end{eqnarray}
where $A$, $B$, $C$, $D$, $E$, $F\in \mathbb{C}$.

\end{theorem}
\begin{proof}
By Remark \ref{r1}, $TSV(a,b)$ is a quadratic Lie conformal algebra corresponding to the Gel'fand-Dorfman bialgebra $V(a,b)$.
Since $L=L\circ L$, $Y=Y\circ L$ and $M=M\circ L$, we obtain that for any $x\in V(a,b)$, there exist $y$, $z\in V(a,b)$ such that
$x=y\circ z$. Therefore, by Proposition \ref{rr1}, we only need to determine all $\alpha_0(x,y)$, $\alpha_1(x,y)$, $\alpha_2(x,y)$ and $\alpha_3(x,y)$
satisfying (\ref{eqqq1})-(\ref{eqq7}) where $x$, $y\in \{L, M, Y\}$. By some computations, we can obtain\\
(1) If $a\neq 1$ and $a\neq 2$ or $b\neq 0$, we obtain (the other bilinear forms vanishing)
\begin{eqnarray*}
&&\alpha_3(L,L)=B,~~\alpha_1(L,L)=A,~~\alpha_1(L,Y)=D,~~\alpha_0(L,Y)=C,~~\alpha_1(Y,Y)=E,\\
&&\alpha_1(L,M)=(a-1)E,~~\alpha_0(L,M)=bE,
\end{eqnarray*}
for any $A$, $B$, $C$, $D$, $E\in \mathbb{C}$ and $aC=bD$;\\
(2) If $a=2$, $b=0$, we obtain (the other bilinear forms vanishing)
\begin{eqnarray*}
&&\alpha_3(L,L)=B,~~\alpha_1(L,L)=A,~~\alpha_3(L,Y)=D,~~\alpha_1(L,Y)=C,\\
&&\alpha_3(Y,Y)=\alpha_3(L,M)=F,~~\alpha_1(Y,Y)=\alpha_1(L,M)=E,
\end{eqnarray*}
for any $A$, $B$, $C$, $D$, $E$, $F\in \mathbb{C}$.\\
(3)If $a=1$, $b=0$, we obtain (the other bilinear forms vanishing)
\begin{eqnarray*}
&&\alpha_3(L,L)=B,~~\alpha_1(L,L)=A,~~\alpha_2(L,Y)=D,~~\alpha_1(L,Y)=C,\\
&&\alpha_1(Y,Y)=E,~~\alpha_0(Y,M)=F,
\end{eqnarray*}
for any $A$, $B$, $C$, $D$, $E$, $F\in \mathbb{C}$.\\

Therefore, by Proposition \ref{rr1}, this theorem holds.
\end{proof}
\end{remark}
\begin{corollary}\label{co}
(1) If $a\neq 1$, $a\neq 2$ and $a$, $b$ are not equal to zero simultaneously or $b\neq 0$, there exists a unique non-trivial central extension of $TSV(a,b)$ by a one-dimensional center $\mathbb{C} \mathfrak{c}$ with the following non-trivial $\lambda$-brackets:
\begin{eqnarray}
&&[L_\lambda L]=(\partial+2\lambda)L+\frac{1}{12}\lambda^3\mathfrak{c},\\
&&[L_\lambda Y]=(\partial+a\lambda+b)Y,\\
&&[Y_\lambda Y]=(\partial+2\lambda)M,~~[L_\lambda M]=(\partial+2(a-1)\lambda+2b)M.
\end{eqnarray}

(2) If $a=b=0$, for any $C$, $D\in \mathbb{C}$ and $(C,D)\neq (0,0)$, there exists a unique non-trivial central extension of $TSV(0,0)$ by a one-dimensional center $\mathbb{C} \mathfrak{c}$ with the following non-trivial $\lambda$-brackets:
\begin{eqnarray}
&&[L_\lambda L]=(\partial+2\lambda)L+\frac{1}{12}\lambda^3\mathfrak{c},\\
&&[L_\lambda Y]=\partial Y+(C+D\lambda)\mathfrak{c},\\
&&[Y_\lambda Y]=(\partial+2\lambda)M,~~[L_\lambda M]=(\partial-2\lambda)M.
\end{eqnarray}

(3) If $a=2$ and $b=0$, for any $D$, $F \in \mathbb{C}$ and $(D,F)\neq (0,0)$, there exists a unique non-trivial central extension of $TSV(2,0)$ by a one-dimensional center $\mathbb{C} \mathfrak{c}$ with the following non-trivial $\lambda$-brackets:
\begin{eqnarray}
&&[L_\lambda L]=(\partial+2\lambda)L+\frac{1}{12}\lambda^3\mathfrak{c},\\
&&[L_\lambda Y]=(\partial+2\lambda)Y+D\lambda^3\mathfrak{c},\\
&&[Y_\lambda Y]=(\partial+2\lambda)M+F\lambda^3 \mathfrak{c},\\
&&[L_\lambda M]=(\partial+2\lambda)M+F\lambda^3\mathfrak{c}.
\end{eqnarray}

(4) If $a=1$ and $b=0$, for any $D$, $F\in \mathbb{C}$, and $(D,F)\neq (0,0)$, there exists a unique non-trivial central extension of $TSV(1,0)$ by a one-dimensional center $\mathbb{C} \mathfrak{c}$ with the following non-trivial $\lambda$-brackets:
\begin{eqnarray}
&&[L_\lambda L]=(\partial+2\lambda)L+\frac{1}{12}\lambda^3\mathfrak{c},\\
&&[L_\lambda Y]=(\partial+\lambda)Y+D\lambda^2\mathfrak{c},~~[L_\lambda M]=\partial M\\
&&[Y_\lambda Y]=(\partial+2\lambda)M,~~[Y_\lambda M]=F\mathfrak{c}.
\end{eqnarray}
\end{corollary}
\begin{proof}
(1) If $b \neq 0$,  by Theorem \ref{th2} (1), replacing $L$, $Y$, $M$ respectively by $L+\frac{A}{2}\mathfrak{c}$, $Y+\frac{C}{b}\mathfrak{c}$,
$M+\frac{E}{2}\mathfrak{c}$, we can let $A=C=D=E=0$. If $b=0$ and $a\neq 0$, $a\neq 1$, $a\neq 2$, similarly, replacing $L$, $Y$, $M$ respectively by $L+\frac{A}{2}\mathfrak{c}$, $Y+\frac{D}{a}\mathfrak{c}$,
$M+\frac{E}{2}\mathfrak{c}$, we can also let $A=C=D=E=0$. Since the central extension is non-trivial, by rescaling $\mathfrak{c}$, we can suppose $B=\frac{1}{12}$. Then, we obtain the result.\\

(2) By  Theorem \ref{th2} (1), replacing $L$,  $M$ respectively by $L+\frac{A}{2}\mathfrak{c}$,
$M+\frac{E}{2}\mathfrak{c}$, we can make $A=E=0$. But, we can not make $C$ and $D$ vanishing. Thus, this result is obtained.\\

(3) and (4) can be similarly obtained.

\end{proof}

\begin{remark}
Define $\alpha$, $\beta$, $\gamma$, $\theta$, $\iota$, $\kappa$, and $\omega$: $TSV(a,b)\otimes TSV(a,b)\rightarrow \mathbb{C}[\lambda]$ by (all other terms are vanishing)
\begin{eqnarray}
\alpha_\lambda(L,L)=\lambda^3,~~\beta_\lambda(L,Y)=1,~~\gamma_\lambda(L,Y)=\lambda,\\
\theta_\lambda(L,Y)=\lambda^2,~~\iota_\lambda(L,Y)=\lambda^3,\\
\kappa_\lambda(Y,Y)=\kappa_\lambda(L,M)=\lambda^3,~~\omega_\lambda(Y,M)=1.
\end{eqnarray}

Then, by Corollary \ref{co} and the cohomology theory of Lie conformal algebra introduced in \cite{BKV} and the above discussion, we obtain
\begin{enumerate}
               \item If $a\neq 1$, $a\neq 2$ and $a$, $b$ are non equal to zero simultaneously or $b\neq 0$,
 $H^2(TSV(a,b),\mathbb{C})=\mathbb{C}\alpha$.\\
               \item $H^2(TSV(0,0),\mathbb{C})=\mathbb{C}\alpha\oplus \mathbb{C}\beta\oplus\mathbb{C}\gamma$.\\
               \item $H^2(TSV(2,0),\mathbb{C})=\mathbb{C}\alpha\oplus \mathbb{C}\iota\oplus\mathbb{C}\kappa$.\\
               \item $H^2(TSV(1,0),\mathbb{C})=\mathbb{C}\alpha\oplus \mathbb{C}\theta\oplus \mathbb{C}\omega$ .
             \end{enumerate}
\end{remark}

\begin{theorem}\label{th1}
Any central extension of $TSV(c)$ by a one-dimensional center $\mathbb{C}\mathfrak{c}$ is of the following form: $\widetilde{TSV(c)}=TSV(c)\oplus \mathbb{C}\mathfrak{c}$ and the non-trivial $\lambda$-brackets are given as follows:
\begin{eqnarray}
&&[L_\lambda L]=(\partial+2\lambda)L+(a_1\lambda+a_3\lambda^3)\mathfrak{c},\\
&&[L_\lambda Y]=(\partial+\frac{3}{2}\lambda+c)Y+(\frac{2}{3}c+\lambda)b_1\mathfrak{c},\\
&&[Y_\lambda Y]=-(\partial+2\lambda)(\partial+2c)M-2c_0\lambda\mathfrak{c},\\
&&[L_\lambda M]=(\partial+2c)M+c_0\mathfrak{c},
\end{eqnarray}
where $a_1$, $a_3$, $b_1$, $c_0\in \mathbb{C}$.
\end{theorem}
\begin{proof}
By (\ref{ef3}), (\ref{ef4}) and (\ref{ec1}), we only need to determine $\alpha_{\lambda}(L,L)$, $\alpha_{\lambda}(L,Y)$, $\alpha_{\lambda}(L,M)$,
$\alpha_{\lambda}(Y,Y)$, $\alpha_{\lambda}(Y,M)$,  $\alpha_{\lambda}(M,M)$.

Replacing $a$, $b$, $c$ by $L$ in (\ref{ec1}), we obtain
\begin{eqnarray}\label{eeq1}
(\lambda+2\mu)\alpha_{\lambda}(L,L)-(\mu+2\lambda)\alpha_{\mu}(L,L)=(\lambda-\mu)\alpha_{\lambda+\mu}(L,L).
\end{eqnarray}
Set $\alpha_\lambda(L,L)=\sum_{i=0}^na_i\lambda^i\in \mathbb{C}[\lambda]$ with $a_n\neq 0$. If $n>1$, by comparing the coefficients of $\lambda^n$,
one gets $(n-3)a_n\mu=0$. Thus, $n=3$. Now, we have $\alpha_{\lambda}(L,L)=a_0+a_1\lambda+a_2\lambda^2+a_3\lambda^3$. By (\ref{ef4}) and (\ref{eeq1}),
we can obtain $\alpha_{\lambda}(L,L)=a_1\lambda+a_3\lambda^3$.

Taking $a=b=L$, $c=Y$ in (\ref{ec1}) and by (\ref{ef3}) and (\ref{ef4}), one has
\begin{gather}\label{eeq2}
(\lambda+\frac{3}{2}\mu+c)\alpha_\lambda(L,Y)-(\mu+\frac{3}{2}\lambda+c)\alpha_\mu(L,Y)\\
=(\lambda-\mu)\alpha_{\lambda+\mu}(L,Y).\nonumber
\end{gather}
Assume that $\alpha_{\lambda}(L,Y)=\sum_{i=0}^mb_i\lambda^i\in \mathbb{C}[\lambda]$ with $b_m\neq 0$.
If $m>1$, by comparing the coefficients of $\lambda^m$ in (\ref{eeq2}), one has $((\frac{5}{2}-m)\mu+c)b_m=0$, which gives $b_m=0$.
Thus, $\alpha_{\lambda}(L,Y)=b_0+b_1\lambda$. Inserting this into (\ref{eeq2}) leads to $b_0=\frac{2}{3}cb_1$ and thus
$\alpha_{\lambda}(L,Y)=b_1(\frac{2}{3}c+\lambda)$.

Letting $a=L$, $b=L$ and $c=M$ in (\ref{ec1}), we get
\begin{eqnarray}\label{eeq3}
(\lambda+2c)\alpha_{\lambda}(L,M)-(\mu+2c)\alpha_{\mu}(L,M)=(\lambda-\mu)\alpha_{\lambda+\mu}(L,M).
\end{eqnarray}
Set $\alpha_{\lambda}(L,M)=\sum_{i=0}^sc_i\lambda^i\in \mathbb{C}[\lambda]$ with $c_s\neq 0$.
If $s>1$, comparing the coefficients of $\lambda^s$, we obtain $2cc_s-c_s(s-1)\mu=0$. Thus, $s\leq 1$.
Then, plugging $\alpha_{\lambda}(L,M)=c_0+c_1\lambda$ into (\ref{eeq3}), we get $cc_1=0$. Thus,
when $c= 0$, $\alpha_{\lambda}(L,M)=c_0+c_1\lambda$. Otherwise, $\alpha_{\lambda}(L,M)=c_0$.

Furthermore, taking $a=L$, $b=c=Y$ in (\ref{ec1}), we obtain
\begin{gather}\label{eeq4}
-(\lambda+2\mu)(\lambda+2c)\alpha_{\lambda}(L,M)-(\mu+\frac{3}{2}\lambda+c)\alpha_{\mu}(Y,Y)\\
=(\frac{1}{2}\lambda-\mu+c)\alpha_{\lambda+\mu}(Y,Y).\nonumber
\end{gather}
Then, we discuss it in two cases, i.e. $c=0$ and $c\neq 0$. If $c=0$, we have known $\alpha_{\lambda}(L,M)=c_0+c_1\lambda$.
Assume that $\alpha_{\lambda}(Y,Y)=\sum_{i=0}^td_i\lambda^i\in \mathbb{C}[\lambda]$ with $d_t\neq 0$. Then,
(\ref{eeq4}) becomes
\begin{gather}\label{eeq5}
-(\lambda+2\mu)\lambda(c_0+c_1\lambda)-(\mu+\frac{3}{2}\lambda)\sum_{i=0}^td_i\mu^i\\
=(\frac{1}{2}\lambda-\mu)\sum_{i=0}^td_i(\lambda+\mu)^i.\nonumber
\end{gather}
Obviously, by (\ref{eeq5}), we get $t\leq 2$. Then, taking $\alpha_{\lambda}(Y,Y)=d_0+d_1\lambda+d_2\lambda^2$ into (\ref{eeq5}) and by
some simple computations, one has $d_0=0$, $d_1=-2c_0$, $d_2=c_1=0$. Thus, $\alpha_{\lambda}(L,M)=c_0$, $\alpha_{\lambda}(Y,Y)=-2c_0\lambda$.
If $c\neq 0$, taking $\alpha_{\lambda}(Y,Y)=\sum_{i=0}^td_i\lambda^i\in \mathbb{C}[\lambda]$ with $d_t\neq 0$ into (\ref{eeq4}) and
with a similar discussion, we obtain $\alpha_{\lambda}(Y,Y)=-2c_0\lambda$.

Finally, replacing $a$, $b$, $c$ by $L$, $M$, $M$ and $L$, $Y$, $M$ respectively in (\ref{ec1}), we get
\begin{eqnarray*}
&&-(\mu+2c)\alpha_{\mu}(M,M)=(-\lambda-\mu+2c)\alpha_{\lambda+\mu}(M,M),\\
&&-(\mu+2c)\alpha_{\mu}(Y,M)=(\frac{1}{2}\lambda-\mu+c)\alpha_{\lambda+\mu}(Y,M).
\end{eqnarray*}
It follows that $\alpha_{\lambda}(M,M)=\alpha_{\lambda}(Y,M)=0$. Then, it is easy to see that (\ref{ec1}) holds in other cases.

By now, the proof is finished.

\end{proof}
\begin{corollary}\label{co1}
If $c\neq 0$, there exists a unique non-trivial central extension of $TSV(c)$ by a one-dimensional center $\mathbb{C} \mathfrak{c}$ with the following non-trivial $\lambda$-brackets:
\begin{eqnarray}
&&[L_\lambda L]=(\partial+2\lambda)L+\frac{1}{12}\lambda^3\mathfrak{c},\\
&&[L_\lambda Y]=(\partial+\frac{3}{2}\lambda+c)Y,~~[L_\lambda M]=(\partial+2c)M,\\
&&[Y_\lambda Y]=-(\partial+2\lambda)(\partial+2c)M.
\end{eqnarray}

If $c=0$, for any $c_0\in \mathbb{C}\backslash\{0\}$, there exists a unique non-trivial central extension of $TSV(0)$ by a one-dimensional center $\mathbb{C} \mathfrak{c}$ with the following non-trivial $\lambda$-brackets:
\begin{eqnarray}
&&[L_\lambda L]=(\partial+2\lambda)L+\frac{1}{12}\lambda^3\mathfrak{c},\\
&&[L_\lambda Y]=(\partial+\frac{3}{2}\lambda)Y,~~[L_\lambda M]=\partial M+c_0 \mathfrak{c},\\
&&[Y_\lambda Y]=-(\partial+2\lambda)\partial M-2c_0\lambda \mathfrak{c}.
\end{eqnarray}
\end{corollary}
\begin{proof}
If $c\neq 0$, by Theorem \ref{th1}, replacing $L$, $Y$, $M$ respectively by $L+\frac{a_1}{2}\mathfrak{c}$, $Y+\frac{2}{3}b_1\mathfrak{c}$,
$M+\frac{c_0}{2c}\mathfrak{c}$, we can make $a_1=b_1=c_0=0$. Since the central extension is nontrivial, by rescaling $\mathfrak{c}$, we can suppose $a_3=\frac{1}{12}$. When $c=0$, by replacing  $L$, $Y$ respectively by $L+\frac{a_1}{2}\mathfrak{c}$, $Y+\frac{2}{3}b_1\mathfrak{c}$, we can let
$a_1=b_0=0$. But, we can not make $c_0$ vanishing. Thus, this result is obtained.
\end{proof}
\begin{remark}
Define $\alpha$ and $\beta: TSV(c)\otimes TSV(c)\rightarrow \mathbb{C}[\lambda]$ by (all other terms are vanishing)
\begin{eqnarray}
\alpha_\lambda(L,L)=\lambda^3,~~\beta_\lambda(L,M)=1,~~\beta_\lambda(Y,Y)=-2\lambda.
\end{eqnarray}

Then, by Corollary \ref{co1}, we get
\begin{enumerate}
               \item If $c\neq 0$, $H^2(TSV(c),\mathbb{C})=\mathbb{C}\alpha$.\\
               \item $H^2(TSV(0),\mathbb{C})=\mathbb{C}\alpha\oplus \mathbb{C}\beta$.\\
             \end{enumerate}
\end{remark}

\section{Conformal derivations and free conformal modules of rank one of $TSV(a,b)$ and $TSV(c)$}
In this section, we will study conformal derivations and free conformal modules of rank one of $TSV(a,b)$ and $TSV(c)$.

First, we investigate conformal derivations of $TSV(a,b)$ and $TSV(c)$.

From now on, denote by $\mathcal{A}$ the ring $\mathbb{C}[\partial]$ of polynomials in the indeterminate $\partial$.

\begin{definition}
A \emph{conformal linear map} between $\mathcal{A}$-modules $U$ and $V$ is a linear map
$\phi_{\lambda}: U\rightarrow \mathcal{A}[\lambda]\otimes_{\mathcal{A}}V$ such that
\begin{eqnarray}
\phi_\lambda(\partial u)=(\partial+\lambda)\phi_\lambda u, ~~\text{for~~all~~$u\in U$.}
\end{eqnarray}
\end{definition}
We will abuse the notation by writing $\phi: U\rightarrow V$ any time it is clear from the context
that $\phi$ is conformal linear.

\begin{definition}
Let $R$ be a Lie conformal algebra. A conformal linear map $d: R\rightarrow R$ is
called a \emph{conformal derivation} of $R$ if
\begin{eqnarray}\label{d1}
d_\lambda[a_\mu b]=[(d_\lambda a)_{\lambda+\mu} b]+[a_\mu(d_\lambda b)], ~~\text{all $a$, $b\in R$.}
\end{eqnarray}
\end{definition}

The space of all conformal derivations of $R$ is denoted by $\text{CDer}(R)$. For any $a\in R$, there is a natural conformal linear map $\text{ad}~a: R\rightarrow R$ such that $$(\text{ad} ~a)_\lambda b=[a_\lambda b],~~~b\in R.$$
All conformal derivations of this kind are called \emph{inner}. The space of all inner conformal derivations is denoted by
$\text{CInn}(R)$.

Next, we begin to determine
conformal derivations of $TSV(a,b)$ and $TSV(c)$.

\begin{theorem}
(1)If $a\neq \frac{3}{2}$, all conformal derivations of $TSV(a,b)$ are inner, i.e. $\text{CDer}(TSV(a,b))=\text{CInn}(TSV(a,b))$.

If $a=\frac{3}{2}$, for any $b_0\in \mathbb{C}\backslash\{0\}$, we can define a non-inner derivation $R$ whose actions are given by
$R_{\lambda}(L)=b_0M$, $R_{\lambda}(M)=R_{\lambda}(Y)=0$. Denote the vector space of all $R$ by $T$.
We get $\text{CDer}(TSV(\frac{3}{2},b))=\text{CInn}(TSV(\frac{3}{2},b))\oplus T$.

(2) All conformal derivations of $TSV(c)$ are inner, i.e. $CDer(TSV(c))=CInn (TSV(c))$.
\end{theorem}
\begin{proof}
(1) Assume that $d$ is a conformal derivation of $TSV(a,b)$. Suppose that
\begin{eqnarray}
&&d_\lambda(L)=A_1(\lambda,\partial)L+B_1(\lambda,\partial)Y+C_1(\lambda,\partial)M,\\
&&d_\lambda(Y)=A_2(\lambda,\partial)L+B_2(\lambda,\partial)Y+C_2(\lambda,\partial)M,\\
&&d_\lambda(M)=A_3(\lambda,\partial)L+B_3(\lambda,\partial)Y+C_3(\lambda,\partial)Y,
\end{eqnarray}
where $A_i(\lambda,\partial)$, $B_i(\lambda,\partial)$, $C_i(\lambda,\partial)\in \mathbb{C}[\lambda,\partial]$ for $i\in\{1,2,3\}$.

Applying $d_\lambda$ to $[L_\mu L]=(\partial+2\mu)L$, the left-hand side becomes
\begin{eqnarray*}
d_\lambda[L_\mu L]&=&[(d_\lambda L)_{\lambda+\mu} L]+[L_\mu (d_\lambda L)]\\
&=&(\partial+2\lambda+2\mu)A_1(\lambda,-\lambda-\mu)L+((a-1)\partial\\
&&+a(\lambda+\mu)-b)B_1(\lambda,-\lambda-\mu)Y+((2a-3)\partial\\
&&+2(a-1)(\lambda+\mu)-2b)C_1(\lambda,-\lambda-\mu)M\\
&&+(\partial+2\mu)A_1(\lambda,\mu+\partial)L+(\partial+a\mu+b)B_1(\lambda,\mu+\partial)Y\\
&&+(\partial+2(a-1)\mu+2b)C_1(\lambda,\mu+\partial)M,
\end{eqnarray*}
while the right side becomes
\begin{eqnarray*}
d_\lambda((\partial+2\mu)L)&=&d_\lambda((\partial+2\mu)L)\\
&=&(\partial+\lambda+2\mu)(A_1(\lambda,\partial)L+B_1(\lambda,\partial)Y+C_1(\lambda,\partial)M).
\end{eqnarray*}
By comparing the coefficients of the similar terms in the above equalities, we obtain
\begin{gather}
(\partial+\lambda+2\mu)A_1(\lambda,\partial)=(\partial+2\lambda+2\mu)A_1(\lambda,-\lambda-\mu)\nonumber\\
\label{w1}+(\partial+2\mu)A_1(\lambda,\mu+\partial),\\
(\partial+\lambda+2\mu)B_1(\lambda,\partial)=((a-1)\partial+a(\lambda+\mu)-b)B_1(\lambda,-\lambda-\mu)\nonumber\\
\label{w2}+(\partial+a\mu+b)B_1(\lambda,\mu+\partial),\\
(\partial+\lambda+2\mu)C_1(\lambda,\partial)=((2a-3)\partial+2(a-1)(\lambda+\mu)-2b)C_1(\lambda,-\lambda-\mu)\nonumber\\
\label{w3}+(\partial+2(a-1)\mu+2b)C_1(\lambda,\mu+\partial).
\end{gather}
Let $A_1(\lambda,\partial)=\sum_{i=0}^nf_i(\lambda)\partial^i$ where $f_i(\lambda)\in\mathbb{C}[\lambda]$ and $f_n(\lambda)\neq 0$.
If $n>1$, comparing the coefficients of $\partial^n$ in (\ref{w1}), we get $(\lambda-n\mu)f_n(\lambda)=0$, obtaining a contradiction.
Thus, $A_1(\lambda,\partial)=f_0(\lambda)+f_1(\lambda)\partial$. Taking this into (\ref{w1}) and by some computations, one can obtain
$f_0(\lambda)=2f_1(\lambda)\lambda$. Hence, $A_1(\lambda,\partial)=f_1(\lambda)(\partial+2\lambda)$. With a similar method, from (\ref{w2}) and
(\ref{w3}), we can obtain that
\begin{eqnarray*}
&&B_1(\lambda,\partial)=g_0(\lambda)+g_1(\lambda)\partial, \text{where $(a-1)g_0(\lambda)=(a\lambda-b)g_1(\lambda)$,}\\
&&C_1(\lambda,\partial)=h_0(\lambda)+h_1(\lambda)\partial, \text{where $(a-\frac{3}{2})h_0(\lambda)=((a-1)\lambda-b)h_1(\lambda)$,}
\end{eqnarray*}
where $g_0(\lambda)$, $g_1(\lambda)$, $h_0(\lambda)$, $h_1(\lambda)\in \mathbb{C}[\lambda]$.

From the above discussion, we can get:\\
If $a\neq 1$ and $a\neq \frac{3}{2}$, one has
\begin{gather}
d_\lambda(L)=f_1(\lambda)(\partial+2\lambda)L+g_1(\lambda)(\frac{a\lambda-b}{a-1}+\partial)Y\nonumber\\
+h_1(\lambda)(\frac{(a-1)\lambda-b}{a-\frac{3}{2}}+\partial)M.
\end{gather}
If $a=1$, one gets
\begin{gather}
d_\lambda(L)=f_1(\lambda)(\partial+2\lambda)L+g_0(\lambda)Y+h_1(\lambda)(2b+\partial)M.
\end{gather}
If $a=\frac{3}{2}$, one has
\begin{gather}
d_\lambda(L)=f_1(\lambda)(\partial+2\lambda)L+g_1(\lambda)((3\lambda-2b)+\partial) Y+h_0(\lambda)M.
\end{gather}

Next, we discuss it in the three cases.

If $a\neq 1$ and $a\neq \frac{3}{2}$, replacing $d$ by $d-{ad}(f_1(-\partial)L+\frac{g_1(-\partial)}{a-1}Y
+\frac{h_1(-\partial)}{2a-3}M)$, we get $d_\lambda(L)=0.$ Then, applying $d_\lambda$ to $[L_\mu Y]=(\partial+a\mu+b)Y$ and using
$d_\lambda(L)=0$, by comparing the coefficients of $L$, $Y$, $M$, one can obtain
\begin{eqnarray}
\label{w4}(\partial+\lambda+a\mu+b)A_2(\lambda,\partial)=(\partial+2\mu)A_2(\lambda,\mu+\partial),\\
\label{w5}(\partial+\lambda+a\mu+b)B_2(\lambda,\partial)=(\partial+a\mu+b)B_2(\lambda,\mu+\partial),\\
\label{w6}(\partial+\lambda+a\mu+b)C_2(\lambda,\partial)=(\partial+2(a-1)\mu+2b)C_2(\lambda,\mu+\partial).
\end{eqnarray}
Comparing the highest degrees of $\lambda$ in (\ref{w4}), (\ref{w5}) and (\ref{w6}), we can directly get
$A_2(\lambda,\partial)=B_2(\lambda,\partial)=C_2(\lambda,\partial)=0$. Thus, $d_\lambda(Y)=0$. Similarly,
applying $d_\lambda$ to $[L_\mu M]=(\partial+2(a-1)\mu+2b)M$, we can obtain  $d_\lambda(M)=0$. Therefore,
in this case, all conformal derivations are inner.

If $a=1$, assume that $g_0(\lambda)=\sum_{i=0}^na_i(\lambda-b)^i$ and let $\gamma(\lambda)=\sum_{i=0}^{n-1}a_{i+1}(\lambda-b)^i$. Then, replacing $d$ by $d-{ad}(f_1(-\partial)L+\gamma(-\partial)Y
-h_1(-\partial)M)$, we get $d_\lambda(L)=a_0Y.$ Then, similar to the first case, applying $d_\lambda$ to $[L_\mu Y]=(\partial+\mu+b)Y$ and using $d_\lambda(L)=a_0Y$, we obtain
\begin{eqnarray}
\label{w7}(\partial+\lambda+\mu+b)A_2(\lambda,\partial)=(\partial+2\mu)A_2(\lambda,\mu+\partial),\\
\label{w8}(\partial+\lambda+\mu+b)B_2(\lambda,\partial)=(\partial+\mu+b)B_2(\lambda,\mu+\partial),\\
\label{w9}(\partial+\lambda+\mu+b)C_2(\lambda,\partial)=a_0(\partial+2\lambda+2\mu)+(\partial+2b)C_2(\lambda,\mu+\partial).
\end{eqnarray}
Comparing the highest degrees of $\lambda$ in (\ref{w7}), (\ref{w8}) and (\ref{w9}), it is easy to see that
$A_2(\lambda,\partial)=B_2(\lambda,\partial)=C_2(\lambda,\partial)=a_0=0$. Thus, $d_\lambda(L)=d_\lambda(Y)=0$. Similarly,
applying $d_\lambda$ to $[L_\mu M]=(\partial+2b)M$, we can obtain  $d_\lambda(M)=0$. Therefore,
in this case, all conformal derivations are also inner.

If $a=\frac{3}{2}$, assume that $h_0(\lambda)=\sum_{i=0}^mb_i(\lambda-2b)^i$ and let $\kappa(\lambda)=\sum_{i=0}^{m-1}b_{i+1}(\lambda-2b)^i$. Then, replacing $d$ by $d-{ad}(f_1(-\partial)L+2g_1(-\partial)Y
+\kappa(-\partial)M)$, we get $d_\lambda(L)=b_0M.$ Then, with a similar discussion as that in the second case, we can get
$d_\lambda(Y)=d_\lambda(M)=0$. Therefore, in this case, $d={ad}(f_1(-\partial)L+2g_1(-\partial)Y
+\kappa(-\partial)M)+R$, where $R$ is defined by $R_\lambda(L)=b_0M,~~R_\lambda(Y)=R_\lambda(M)=0$.
Moreover, $R$ is a non-inner conformal derivation for any $b_0\in \mathbb{C}\backslash\{0\}$.

By now, the proof is finished.

(2)  Assume that $d$ is a conformal derivation of $TSV(c)$.
Let $d_\lambda(L)=E_1(\lambda,\partial)L+F_1(\lambda,\partial)Y+G_1(\lambda,\partial)M$. Similar to that in (1), by applying
$d_\lambda$ to $[L_\mu L]=(\partial+2\mu)L$, we can obtain
\begin{gather}
(\partial+\lambda+2\mu)E_1(\lambda,\partial)=(\partial+2\lambda+2\mu)E_1(\lambda,-\lambda-\mu)\nonumber\\
\label{w10}+(\partial+2\mu)E_1(\lambda,\mu+\partial),\\
(\partial+\lambda+2\mu)F_1(\lambda,\partial)=(\frac{1}{2}\partial+\frac{3}{2}(\lambda+\mu)-c)F_1(\lambda,-\lambda-\mu)\nonumber\\
\label{w11}+(\partial+\frac{3}{2}\mu+c)F_1(\lambda,\mu+\partial),\\
\label{w12}(\partial+\lambda+2\mu)G_1(\lambda,\partial)=(\partial+2c)(G_1(\lambda,\mu+\partial)-G_1(\lambda,-\lambda-\mu)).
\end{gather}
With the similar discussion that in (1), we can get
\begin{eqnarray}
&&E_1(\lambda,\partial)=f(\lambda)(\partial+2\lambda),~~~E_1(\lambda,\partial)=g(\lambda)(\partial+3\lambda-2c),\\
&&G_1(\lambda,\partial)=h(\lambda)(\partial+2c),
\end{eqnarray}
where $f(\lambda)$, $g(\lambda)$ and $h(\lambda)\in \mathbb{C}[\lambda]$.

Replacing $d$ by $d-ad(f(-\partial)L+2g(-\partial)Y-h(-\partial)M)$, we can get $d_\lambda (L)=0$.
Then, similar to that in (1), by applying $d_\lambda$ to $[L_\mu Y]=(\partial+\frac{3}{2}\mu+c)Y$ and $[L_\mu M]=(\partial+2c)M$ respectively and using
$d_\lambda(L)=0$, one can obtain $d_\lambda(Y)=d_\lambda(M)=0$. Therefore, all conformal derivations of $TSV(c)$ are inner.
\end{proof}

Finally, we present a characterization of all free non-trivial $TSV(a,b)$-modules and $TSV(c)$-modules of rank one.
\begin{theorem}
(1) If $a\neq 1$ or $b\neq 0$, all free non-trivial $TSV(a,b)$-modules of rank one over $\mathbb{C}[\partial]$ are as follows:
\begin{eqnarray*}
M_{\alpha,\beta}=\mathbb{C}[\partial]v,~~L_\lambda v=(\partial+\alpha\lambda+\beta)v,~~Y_\lambda v=M_\lambda v=0, \text{for~~some~~$\alpha$, $\beta\in \mathbb{C}$.}
\end{eqnarray*}

In addition, all free non-trivial $TSV(1,0)$-modules of rank one over $\mathbb{C}[\partial]$ are as follows:
\begin{eqnarray*}
M_{\alpha,\beta}=\mathbb{C}[\partial]v,~~L_\lambda v=(\partial+\alpha\lambda+\beta)v,~~Y_\lambda v=\gamma v,~~M_\lambda v=0, \text{for~~some~~$\alpha$, $\beta$, $\gamma\in \mathbb{C}$.}
\end{eqnarray*}

(2)All free non-trivial $TSV(c)$-modules of rank one over $\mathbb{C}[\partial]$ are as follows:
\begin{eqnarray*}
M_{\alpha,\beta}=\mathbb{C}[\partial]v,~~L_\lambda v=(\partial+\alpha\lambda+\beta)v,~~Y_\lambda v=M_\lambda v=0, \text{for~~some~~$\alpha$, $\beta\in \mathbb{C}$.}
\end{eqnarray*}

\end{theorem}
\begin{proof}
(1) Assume that
\begin{eqnarray}\label{h5}
L_\lambda v=f(\lambda,\partial)v, ~~Y_\lambda v=g(\lambda,\partial)v,~~\text{and} ~~M_\lambda v=h(\lambda,\partial)v,
\end{eqnarray}
where
$f(\lambda,\partial)$, $g(\lambda,\partial)$, $h(\lambda,\partial)\in \mathbb{C}[\lambda,\partial]$.

Since $\mathbb{C}[\partial]L$ is the Virasoro Lie conformal algebra, according to Proposition \ref{PP1}, we get that $f(\lambda,\partial)=0$ or
$f(\lambda,\partial)=\partial+\alpha\lambda+\beta$ for some $\alpha$, $\beta\in \mathbb{C}$. Since $[Y_\lambda M]=0$, we have
$Y_\mu(M_\lambda v)=M_\lambda(Y_\mu v)$. Therefore, we can obtain $h(\lambda,\partial+\mu)g(\mu,\partial)=g(\mu,\partial+\lambda)h(\lambda,\partial)$, which implies $\text{deg}_\lambda h(\lambda,\partial)
=\text{deg}_\partial g(\lambda,\partial)+\text{deg}_\lambda h(\lambda,\partial)$, where $\text{deg}_\lambda h(\lambda,\partial)$ is
the highest degree of $\lambda$ in $h(\lambda,\partial)$. Thus, $deg_\partial g(\lambda,\partial)=0$, i.e. $g(\lambda,\partial)=A(\lambda)$
for some $A(\lambda)\in \mathbb{C}[\lambda]$. Therefore, $h(\lambda,\partial)=B(\lambda)$ for some $B(\lambda)\in \mathbb{C}[\lambda]$.
Then, considering $[Y_\lambda Y]_{\lambda+\mu}v=((\partial+2\lambda)M)_{\lambda+\mu}v$, we get
\begin{eqnarray}
\label{h1}g(\mu,\partial+\lambda)g(\lambda,\partial)-g(\lambda,\partial+\mu)g(\mu,\partial)=(\lambda-\mu)h(\lambda+\mu,\partial).
\end{eqnarray}
Taking $g(\lambda,\partial)=A(\lambda)$ and $h(\lambda,\partial)=B(\lambda)$ into (\ref{h1}), we have $(\lambda-\mu)B(\lambda)=0$.
Therefore, $h(\lambda,\partial)=B(\lambda)=0$, i.e. $M_\lambda v=0$. Thus, $[L_\lambda M]_{\lambda+\mu}v=((\partial+2(a-1)\lambda+2b)M)_{\lambda+\mu}v=0$, $[Y_\lambda M]_{\lambda+\mu}v=0$ and $[M_\lambda M]_{\lambda+\mu}v=0$ hold.
Finally, by $[L_\lambda Y]_{\lambda+\mu}v=((\partial+a\lambda+b)Y)_{\lambda+\mu}v$, we can get
\begin{gather}
\label{h2}(-\lambda-\mu+a\lambda+b)g(\lambda+\mu,\partial)=g(\mu,\lambda+\partial)f(\lambda,\partial)\\
-g(\mu,\partial)f(\lambda,\mu+\partial).\nonumber
\end{gather}
Obviously, if $f(\lambda,\partial)=0$, $g(\lambda,\partial)=0$, which means this module action is trivial. Therefore,
$f(\lambda,\partial)=\partial+\alpha\lambda+\beta$. Taking this and $g(\lambda,\partial)=A(\lambda)$ into (\ref{h2}), we obtain
\begin{eqnarray}
\label{h3}(-\mu+(a-1)\lambda+b)A(\lambda+\mu)=-A(\mu)\mu.
\end{eqnarray}
Obviously, $A(\mu)=\gamma$ for some $\gamma\in \mathbb{C}$. Then, plugging it into (\ref{h3}), we can get:
If $a\neq 1$ or $b\neq 0$, $g(\lambda,\partial)=\gamma=0$; if $a=1$ and $b=0$, $g(\lambda,\partial)=\gamma$ for any
$\gamma\in \mathbb{C}$.

Now, the proof is finished.

(2) Assume that the module action of $TSV(c)$ on $\mathbb{C}[\partial]v$ is defined by (\ref{h5}). According to that in (1), with the same discussions
on $[L_\lambda L]_{\lambda+\mu}v=((\partial+2\lambda)L)_{\lambda+\mu}v$ and $[Y_\lambda M]_{\lambda+\mu}v=0$,
we get that $f(\lambda,\partial)=0$ or
$f(\lambda,\partial)=\partial+\alpha\lambda+\beta$, $g(\lambda,\partial)=A(\lambda)$ and $h(\lambda,\partial)=B(\lambda)$ for some $\alpha$, $\beta\in \mathbb{C}$ and $A(\lambda)$, $B(\lambda)\in \mathbb{C}[\lambda]$.

Applying $[Y_\lambda Y]=-(\partial+2\lambda)(\partial+2c)M$ to $v$, we can obtain
$Y_\lambda(Y_\mu v)-Y_\mu(Y_\lambda v)=(-(\partial+2\lambda)(\partial+2c)M)_{\lambda+\mu}v$. From this,
we can deduce that $
g(\lambda,\partial)g(\mu,\lambda+\partial)-g(\mu,\partial)g(\lambda,\mu+\partial)=-(\lambda-\mu)(-\lambda-\mu+2c)h(\lambda+\mu,\partial)$.
Taking $g(\lambda,\partial)=A(\lambda)$ and $h(\lambda,\partial)=B(\lambda)$ into the above equality, we get
$-(\lambda-\mu)(-\lambda-\mu+2c)B(\lambda+\mu)=0$. Thus, $h(\lambda,\partial)=B(\lambda)=0$. Then,
by $[L_\lambda Y]_{\lambda+\mu}v=((\partial+\frac{3}{2}\lambda+c)Y)_{\lambda+\mu}v$, with a similar discussion as that in (1),
we have $f(\lambda,\partial)=\partial+\alpha\lambda+\beta$ and $g(\lambda,\partial)=A(\lambda)=0$. Then, $[L_\lambda M]_{\lambda+\mu}v=((\partial+2c)M)_{\lambda+\mu}v=0$,$[Y_\lambda M]_{\lambda+\mu}v=0$ and $[M_\lambda M]_{\lambda+\mu}v=0$ hold,
which concludes the proof.
\end{proof}

\end{document}